\documentclass[reqno]{amsart}
\usepackage{amssymb}
\usepackage{amsmath}
\usepackage{graphicx}
\usepackage{dcolumn}
\usepackage{hyperref}
\usepackage{txfonts}
\usepackage{amsfonts}
\usepackage[mathscr]{eucal}
\newtheorem{theorem}{Theorem}[section]
\newtheorem{corollary}[theorem]{Corollary}

\newtheorem{proposition}[theorem]{Proposition}
\theoremstyle{definition}

\newtheorem{example}[theorem]{Example}

\newtheorem*{conjecture}{Conjecture}
\numberwithin{equation}{section}
\theoremstyle{remark}
\newtheorem{remark}[theorem]{Remark}

\begin{document}
\title{Maximal lexicographic spectra and ranks for states with fixed uniform margins}

\author{Xin Li}
\address{Department of Mathematics, Zhejiang University of Technology, Hangzhou 310023, P. R. China}
\email{xinli1019@126.com}

\thanks{The research is supported by National Natural Science Foundation of China (Grant No.11626211).}



\begin{abstract}

We find the spectrum in maximal lexicographic order for quantum states $\rho_{AB}\in\mathcal{H}_A\otimes \mathcal{H}_B$
with margins $\rho_A=\frac{1}{n}I_n$ and $\rho_B=\frac{1}{m}I_m$ and discuss the construction of $\rho_{AB}$. By nonzero rectangular Kronecker coefficients, we give counterexamples for Klyachko's conjecture which says that a quantum state with maximal lexicographical spectrum has minimal rank among all states with given margins. Moreover, we show that  quantum states with the maximal lexicographical spectrum are extreme points.

%
\end{abstract}

\maketitle

\section{Introduction}



The quantum marginal problem is about relations between spectrum of mixed
state $\rho_{AB}$ of two (or multi) component system $\mathcal{H}_{AB}=\mathcal{H}_A\otimes \mathcal{H}_B$ and that of reduced
states $\rho_{A}$ and $\rho_{B}$ \cite{Chris05,Daft,Kly04,Kly06}. As margins of a pure state are isospectral, for $Spec~\rho_A \neq Spec~\rho_B$ state $\rho_{AB}$ can't be pure. It is interesting to measure the closeness between $\rho_{AB}$ and the pure states.
A state $\rho$ is pure if and only if its maximal eigenvalue is equal to
one. Hence the maximal eigenvalue may be considered as a measure of purity. On the other hand, a state $\rho$ is pure if and only if its rank equals to one. So pure states can be also characterized by their rank.
In \cite[Sec.6.4]{Kly04}, Klyachko raised the following conjecture:
\begin{conjecture}
State $\rho_{AB}$ with maximal lexicographical spectrum has minimal
rank among all states with given margins $\rho_A$, $\rho_B$.
\end{conjecture}

Let $\mathcal{C}(\frac{1}{n}I_n,\frac{1}{m}I_m)$ denote the convex set of states with margins $\frac{1}{n}I_n$, $\frac{1}{m}I_m$, where $I_n$, $I_m$ are identity matrices of size $n$ and $m$. Since the spectra of $\frac{1}{n}I_n$ and $\frac{1}{m}I_m$ are uniform probability distributions, we call them uniform margins. Motivated by Klyachko's conjecture, in this paper we study the maximal lexicographic spectrum and ranks of states in $\mathcal{C}(\frac{1}{n}I_n,\frac{1}{m}I_m)$. We give counterexamples for Klyachko's  conjecture, and show that there exist states which have the maximal lexicographic spectrum, but they don't have the minimal rank.
Moreover, we discuss how to construct the states in  $\mathcal{C}(\frac{1}{n}I_n,\frac{1}{m}I_m)$ with prescribed ranks, which generalize the construction in \cite{BCI11}. Our discussion is based on the correspondence between Kronecker coefficients and the spectra of density operators \cite{Chris05,Chris06,Chris07, Kly04}.

The paper is organized as follows. In Section 2 we give the  definitions and results used in the paper.  In Section 3 we construct the maximal lexicographic spectrum of states in $\mathcal{C}(\frac{1}{n}I_n,\frac{1}{m}I_m)$. We provide two classes of counterexamples for Klyachko's conjecture and show that states with the maximal lexicographic spectrum are extreme points. In Section 4 we give the construction of states with prescribed ranks in $\mathcal{C}(\frac{1}{n}I_n,\frac{1}{m}I_m)$.

\section{Preliminaries}\label{se:preli}

\subsection{Partitions and Kronecker coefficients}
A partition $\lambda$ of $n\in \mathbb{N}$ is a monotonically decreasing sequence $\lambda=(\lambda_1,\lambda_2, \ldots,\lambda_k) $ of natural
numbers such that $\sum_{i=1}^k\lambda_i=n$ and denoted by $\lambda\vdash n$. The length $l(\lambda)$ of $\lambda$ is defined as the
number of its nonzero parts and its size as $|\lambda|:=\sum_{i=1}^k \lambda_i$. If $\lambda_1=\lambda_2= \cdots=\lambda_k$, we call $\lambda$ a rectangular partition. The normalization $ \bar{\lambda}:= \lambda/n = (\lambda_1/n,\lambda_2/n, \ldots,\lambda_k/n)$ defines a probability distribution on $\mathbb{N}$. The Young diagram of a partition $\lambda$ is a top-aligned and left-aligned array of boxes such that in row $i$ we have $\lambda_i$ boxes. If we transpose a Young diagram at the main diagonal we obtain another Young diagram, the corresponding partition is denoted by $\lambda^t$. For  $\ell\in \mathbb{N}$, we let $\ell\lambda$ stand
for the partition arising by multiplying all components of $\lambda$ by $\ell$. If $\mu= (\mu_1, \mu_2. . . )$ is another partition, we denote by $\lambda\cap\mu=(\min\{\lambda_1,\mu_1\},\min\{\lambda_2,\mu_2\},...)$ which is also a partition.

Let  $\chi^\lambda$, $\chi^\mu$ denote the complex irreducible characters of the symmetric group $S_n$ corresponding to the partitions $\lambda,\mu$ of $n$. Their Kronecker product $\chi^\lambda\otimes\chi^\mu$ is also a
character of $S_n$. The Kronecker coefficient $g(\lambda,\mu;\nu)$ associated with three partitions $\lambda,\mu,\nu$ of $n$ is defined as the multiplicity  of $\chi^\nu$ in $\chi^\lambda\otimes\chi^\mu$, that is, the coefficient of $\chi^\nu$ in the expansion
$$\chi^\lambda\otimes\chi^\mu=\sum_{\nu\vdash n}g(\lambda,\mu;\nu)\chi^\nu.$$
In above, all partitions corresponding to the set of nonzero Kronecker coefficients is denoted by
$$\Phi(\lambda,\mu)=\{\nu~|~g(\lambda,\mu;\nu)\neq0\}.$$
Kronecker coefficients are only understood in some special cases. It is a difficult open problem to give a combinatorial interpretation of the numbers $g(\lambda,\mu;\nu)$ \cite{Kly04,Stan99}.


Let $\lambda\vdash n$, $\mu\vdash m$, $\nu\vdash n+m$ and $\chi^\lambda\hat{\otimes}\chi^\mu$ be the outer product of $\chi^\lambda$ and $\chi^\mu$. The Littlewood-Richardson coefficient $c_{\lambda,\mu}^\nu$ is the
multiplicity  of $\chi^\nu$ in $\chi^\lambda\hat{\otimes}\chi^\mu$. There is an efficient algorithm for calculation $c_{\lambda,\mu}^\nu$ known as Littlewood-Richardson rule, see \cite{Jam81,Stan99} for details. By the semigroup property of Littlewood-Richardson coefficients (see for example \cite{Chris06}), we have that if $c_{\lambda,\mu}^\nu>0$ then $c_{\ell\lambda,\ell\mu}^{\ell\nu}>0$ for all $\ell>0$.

\subsection{Spectra of quantum states and their orders}
Let $\mathcal{H}$ be a $d$-dimensional complex Hilbert space and denote by $\mathcal{L}(\mathcal{H})$ the space of linear operators mapping $\mathcal{H}$ into itself. A positive semidefinite operator $\rho\in \mathcal{L}(\mathcal{H})$ is called a density operator if $tr(\rho)=1$. Denote the set of density operators in $\mathcal{L}(\mathcal{H})$ by $\mathcal{D}(\mathcal{H})$. Density operators are the mathematical formalism to describe the states of quantum objects. Denote the spectrum of $\rho$ by $Spec~\rho$, it will always be understood as the vector $(r_1, . . . , r_d )$ of eigenvalues of $\rho$ in decreasing order, that is, $r_1\geq . . . \geq r_d$. The rank of a density operator $\rho$ is denoted by $rank~\rho$.

Suppose that $\lambda=(\lambda_1,\lambda_2,...,\lambda_{d_1})$ and $\mu=(\mu_1,\mu_2...,\mu_{d_2})$ are  spectra of two quantum states or partitions of some $n\in \mathbb{N}$.
Recall that $\lambda$ is less than $\mu$ in lexicographic order if, for some index $i$,
$$\lambda_j=\mu_j~\text{for}~j<i~\text{and}~\lambda_i<\mu_i,$$
which is denoted by $\lambda\leq\mu$. On the other hand, $\lambda$ is less than $\mu$ in dominance order (or $\lambda$ is majorized by $\mu$) if
$$\sum_{i=1}^{k}\lambda_i\leq\sum_{i=1}^k\mu_i~\text{for all }~k\geq1,$$
which is denoted by $\lambda\unlhd\mu$. It is not hard to see that if $\lambda\unlhd\mu$ then we have $\lambda\leq\mu$, that is, lexicographic order is a refinement of the dominance order \cite{Sag01}.


The state of a system composed of particles $A$ and $B$ is described by a density operator on a tensor product of two Hilbert spaces, $\rho_{AB}\in \mathcal{D}(\mathcal{H}_A\otimes \mathcal{H}_B)$. The partial trace $\rho_A = tr_B(\rho_{AB})$ $\in \mathcal{D}(\mathcal{H}_A)$ of $\rho_{AB}$ obtained by tracing over $B$ then defines the state of particle $A$. Similarly,  $\rho_B=tr_A(\rho_{AB})$ is obtained by tracing out the subsystem $A$. In this way, $\rho_A$, $\rho_B$ are called marginal states (or margins) of $\rho_{AB}$ \cite{Kly04}. For any two density operators $\rho_A\in\mathcal{D}(\mathcal{H}_A)$, $\rho_B\in\mathcal{D}(\mathcal{H}_B)$, the set of states in $\mathcal{D}(\mathcal{H}_A\otimes \mathcal{H}_B)$ with margins $\rho_A$ and $\rho_B$ is defined as
$$\mathcal{C}(\rho_A, \rho_B):=\{\rho\in \mathcal{D}(\mathcal{H}_A\otimes\mathcal{H}_B)|tr_B(\rho)=\rho_A,
~tr_A(\rho)=\rho_B\}.$$
The set of spectra of states in $\mathcal{C}(\rho_A, \rho_B)$ is defined as
$$\mathcal{S}(\rho_A, \rho_B):=\{Spec~\rho|\rho\in \mathcal{C}(\rho_A, \rho_B)\}.$$
It was shown in \cite{Chris05,Chris07,Kly04} that $\mathcal{S}(\rho_A, \rho_B)$ is a convex polytope. Hence, Klyachko's conjecture states that if the spectrum of a state in $\mathcal{C}(\rho_A, \rho_B)$ has maximal lexicographic order in $\mathcal{S}(\rho_A, \rho_B)$, then it has minimal rank among all other states in $\mathcal{C}(\rho_A, \rho_B)$.
\subsection{The spectra and nonzero Kronecker coefficients}
Given a description of the set of possible triples of spectra $(Spec~\rho_{AB}, Spec~\rho_A, Spec~\rho_B)$ for fixed $d_A =\dim \mathcal{H}_A$ and $d_B=\dim \mathcal{H}_B$ is fundamental in quantum marginal problems.

It turns out that the admissible spectral triples correspond to nonzero Kronecker coefficients. It was shown in \cite{Kly04} (see also \cite{Chris05,Chris07}) that for a density operator $\rho_{AB}$ with the rational spectral triple $(Spec~\rho_A, Spec~\rho_B,$ $Spec~\rho_{AB})$ $=$ $(r_A, r_B, r_{AB})$ there is an integer $m > 0$ such that $g(mr_A,mr_B;mr_{AB})\neq 0$. Conversely, suppose that $\lambda$, $\mu$, $\nu\vdash k$ are partitions with lengths $l(\lambda)\leq m$, $l(\mu)\leq n$, $l(\nu)\leq nm$.  In \cite{Chris07} the authors showed that if $g(\lambda,\mu;\nu)\neq 0$ then there exists a density operator $\rho_{AB}$ on $\mathcal{H}_A \otimes \mathcal{H}_B =\mathbb{C}^m \otimes \mathbb{C}^n$ with spectra
$Spec~\rho_A = \bar{\lambda}$, $Spec~\rho_B = \bar{\mu}$, $Spec~\rho_{AB}=\bar{\nu}$. Hence the length of $\nu$ is the rank of $\rho_{AB}$.





%

\section{The maximal lexicographic spectrum of $\mathcal{C}(\frac{1}{n}I_n,\frac{1}{m}I_m)$ and counterexamples for Klyachko's conjecture}
\label{sec:results}

In this section, through the correspondence between the spectra  and nonzero Kronecker coefficients we will find the maximal lexicographic spectrum for states in $\mathcal{C}(\frac{1}{n}I_n,\frac{1}{m}I_m)$ and give two classes of counterexamples for Klyachko's conjecture. Moreover, we discuss their extremity.
The following proposition is well-known (see e.g. \cite{Iken17}).

\begin{proposition}[Transposition property]\label{prp-trans}
Suppose that $\lambda$, $\mu$, $\nu\vdash n$. Then we have
$g(\lambda,\mu;\nu)=g(\lambda,\mu^t;\nu^t)$.
\end{proposition}

By the discussion in Section 6.4 of \cite{Kly04} we have the following proposition which is also well-known. It gives a lower bound for ranks of states in $\rho\in\mathcal{C}(\rho_A, \rho_B)$. In many cases the lower bound is best, see Remark \ref{re-attain}.
\begin{proposition}\label{prp-rankab}
Let $\rho\in\mathcal{C}(\rho_A, \rho_B)$ and denote $k_A=rank~\rho_A$, $k_B=rank~\rho_B$. Suppose that $k_A\leq k_B$. Then $\lceil\frac{k_B}
{k_A}\rceil\leq rank~\rho\leq k_A k_B$.
\end{proposition}

\begin{proposition}\label{prp-sabphi}
Suppose that $\bar{\lambda}=Spec~\rho_A$, $\bar{\mu}=Spec~\rho_B$ are rational spectra and $n_{AB}$ the minimal positive integer such that $\lambda=n_{AB}\bar{\lambda}$ and $\mu=n_{AB}\bar{\mu}$ are partitions. Then $\mathcal{S}(\rho_A, \rho_B)$ is the closure of $\frac{1}{\ell n_{AB}}\Phi(\ell\lambda,\ell\mu)$ for $\ell\geq1$, that is,
$$\mathcal{S}(\rho_A, \rho_B)=\overline{\bigcup_{\ell=1}^{\infty}\frac{1}{\ell n_{AB}}\Phi(\ell\lambda,\ell\mu)}.$$
\end{proposition}
\begin{proof}
Let $\mathcal{QS}(\rho_A, \rho_B)$ be the set of rational spectra in $\mathcal{S}(\rho_A, \rho_B)$. It suffices to show that
\begin{align}\label{eq-qsab}
\mathcal{QS}(\rho_A, \rho_B)=\bigcup_{\ell=1}^{\infty}\frac{1}{\ell n_{AB}}\Phi(\ell\lambda,\ell\mu).
\end{align}

By Theorem 2.3 of \cite{Chris07}, for any $\bar{\nu}\in \mathcal{QS}(\rho_A, \rho_B)$ there exists an integer $m$ such that $m\bar{\nu}$, $m\bar{\lambda}$ and $m\bar{\mu}$ are partitions and
\begin{align}\label{eq-mlam}
g(m\bar{\lambda},m\bar{\mu};m\bar{\nu})\neq 0.
\end{align}
Since $\bar{\lambda}$, $\bar{\mu}$ consist of rational numbers, let
$$\bar{\lambda}=(a_1/b_1, a_2/b_2,\ldots,a_s/b_s)~\text{and}~\bar{\mu}=(c_1/d_1, c_2/d_2,\ldots,c_t/d_t)$$ where $a_i$ and $b_i$ ($i=1,,2\ldots,s$), $c_j$ and $d_j$ ($j=1,,2\ldots,t$) are integers and relatively prime. Then we have that $n_{AB}$ (resp. $m$) is the least common multiple (resp. the common multiple) of $\{b_1,b_2,\ldots,b_s,d_1,d_2,\ldots,d_t\}$. Hence we have $n_{AB}\mid m$. Let $k=n_{AB}\mid m$, then (\ref{eq-mlam}) is equivalent to
$$\bar{\nu}\in \frac{1}{k n_{AB}}\Phi(k\lambda,k\mu).$$
Thus we have that
$\mathcal{QS}(\rho_A, \rho_B)\subseteq\bigcup_{\ell=1}^{\infty}\frac{1}{\ell n_{AB}}\Phi(\ell\lambda,\ell\mu),$ and therefore (\ref{eq-qsab}) holds by Theorem 3.2 of \cite{Chris07}.
\end{proof}

Given $\lambda$ and $\mu$ partitions such that  $\mu_i \leq\lambda_i$ for all $i\geq1$, we write $\mu\subseteq\lambda$ (or $\mu\subset\lambda$ if $\mu_i < \lambda_i$ for some $i$). In \cite{Val03} (see also \cite{Cla93}), the author introduced a construction, which can be used to obtain the maximal component, in the lexicographic order in $\chi^\lambda\otimes\chi^\mu$. The construction is as follows.


Let $\lambda, \mu$ be partitions of $n$, together with two strictly decreasing sequences of partitions
\begin{equation}\label{eq-seqlm}
\begin{split}
\lambda&=\lambda(1)\supset\cdots\lambda(r)\supset\lambda(r+1)=\emptyset,\\
\mu&=\mu(1)\supset\cdots\mu(r)\supset\mu(r+1)=\emptyset,
\end{split}
\end{equation}
such that
$$c_{\lambda(i)\cap\mu(i),\lambda(i+1)}^{\lambda(i)}\neq 0~~\text{and}~~c_{\lambda(i)\cap\mu(i),\mu(i+1)}^{\mu(i)}\neq 0,$$
for all $1\leq i \leq r$. Set
$$\nu_i =|\lambda(i) \cap \mu(i)|$$
for all $1 \leq i \leq r$. Then $\nu= (\nu_1, . . . ,\nu_r )$ is a partition of $n$.
Any $\nu$ obtained in this
way is called a {\it partition of strip type derived from} $ (\lambda, \mu)$ \cite{Val03}. For example, if we let $\lambda=(2^5)$, $\mu=(5^2)$ and $\nu=(4,4,1,1)$, then $\nu$ is a partition of strip type derived from $ (\lambda, \mu)$. The corresponding sequences of partitions are
\begin{equation*}
\begin{split}
&(2^5)\supset (2^3)\supset (2)
\supset (1)\supset \emptyset,\\
 &(5^2)\supset (3^2)\supset (1^2)
\supset (1)\supset \emptyset.\\
\end{split}
\end{equation*}
 Clausen and Meier showed that the maximal component $\chi^\nu$ of $\chi^\lambda\otimes\chi^\mu$ in the lexicographic order corresponds to a derived partition of strip type \cite{Cla93}.


%

Observe that $\lambda\cap\mu\subseteq\lambda$. In the Young diagram of $\lambda$ we let $\lambda\backslash \lambda\cap\mu$ denote boxes which belong to  $\lambda$ but not $\lambda\cap\mu$ (similarly for $\mu\backslash \lambda\cap\mu$). It is called \emph{skew diagram} in \cite{Val03} (see also \cite{Bess14}).  $\lambda\backslash \lambda\cap\mu$ may correspond to a partition. For example, if we let $\lambda=(2^5)$ and $\mu=(5^2)$, then $\lambda\cap\mu=(2,2)$ and $\lambda\backslash \lambda\cap\mu=(2^3)$ which is also a partition.

%

%
%

\begin{proposition}\label{prp-only1}
Suppose that $\lambda$, $\mu\vdash n$ are two rectangular partitions. Then there is exactly one partition of strip type derived from $ (\lambda, \mu)$ which has the maximal lexicographic order in $\Phi(\lambda,\mu)$.
\end{proposition}
\begin{proof}
Since  $\lambda$ and $\mu$ are two rectangular partitions, by Lemma 3.3 of~\cite{Bess14} there exists only one pair of partitions $\lambda(2)$, $\mu(2)$ such that $c^{\lambda}_{\lambda\cap\mu,\lambda(2)}\neq0$ and $c^{\mu}_{\lambda\cap\mu,\mu(2)}\neq0$ where  $\lambda(2)=\lambda\backslash \lambda\cap\mu$ and $\mu(2)=\mu\backslash \lambda\cap\mu$ are also rectangular partitions.
Similarly, if we continue the construction described in (\ref{eq-seqlm}), for each $1\leq i\leq r$ there exists only one pair of rectangular partitions $\lambda(i+1)$, $\mu(i+1)$ such that $c^{\lambda(i)}_{\lambda(i)\cap\mu(i),\lambda(i+1)}\neq0$ and $c^{\mu(i)}_{\lambda(i)\cap\mu(i),\mu(i+1)}\neq0$.
Hence, there is exactly one partition of strip type derived from $ (\lambda, \mu)$ denoted by $\nu$. By Theorem 3.5 of \cite{Val03} we have that $\nu\in\Phi(\lambda,\mu)$.
Since the maximal component of $\chi^\lambda\otimes\chi^\mu$ in the lexicographic order corresponds to a derived partition of strip type, by uniqueness we have that $\nu$ has the maximal lexicographic order in $\Phi(\lambda,\mu)$ \cite{Cla93,Val03}.
\end{proof}
%

The following proposition can be  obtained from (6.10) of \cite{Kly04} which gives us the value of Kronecker coefficient in Proposition \ref{prp-only1}.
\begin{proposition}\label{prp-multiv}
Suppose that $\lambda$, $\mu\vdash n$ are two rectangular partitions. Let $\nu$ be the  partition of strip type derived from $ (\lambda, \mu)$. Then $g(\lambda,\mu;\nu)=1$.
\end{proposition}


In the following, we let $lcm(n,m)$ denote the least common multiple of $n$ and $m$.
\begin{theorem}\label{thm-maxlex}
For $n\leq m$, let $\lambda=(a^n)$, $\mu=(b^m)\vdash k$ be two rectangular partitions, where $k=lcm(n,m)$, $a=n\mid k$ and $b=m\mid k$. Suppose that $\nu$ is the  partition of strip type derived from $ (\lambda, \mu)$.  Then  the maximal lexicographic spectrum for states in $\mathcal{C}(\frac{1}{n}I_n,\frac{1}{m}I_m)$ is $\nu/k$.
\end{theorem}

\begin{proof}
It is equivalent to show that $\nu/k$ is maximal in the lexicographic order in $\mathcal{S}(\frac{1}{n}I_n,\frac{1}{m}I_m)$.
For $k=lcm(n,m)$, it is the minimal integer such that both $a=k/n$ and $b=k/m$ are integers. For $\lambda=(a^n)$, $\mu=(b^m)\vdash k$,
by Proposition \ref{prp-sabphi} we have that
$$\mathcal{S}(\frac{1}{n}I_n,\frac{1}{m}I_m)
=\overline{\bigcup_{\ell=1}^{\infty}\frac{1}{\ell k}\Phi(\ell\lambda,\ell\mu)}.$$

Let $\nu^{(1)}=\nu$ and $\nu^{(\ell)}$ be the  partition of strip type derived from $ (\ell\lambda, \ell\mu)$ for all $\ell\geq1$. By (\ref{eq-seqlm}),   let $\nu^{(1)}=(\nu^{(1)}_1,\nu^{(1)}_2,\ldots,\nu^{(1)}_r)$ be derived from  the following two strictly decreasing sequences of partitions
\begin{equation}\label{eq-seqlm1}
\begin{split}
\lambda&=\lambda(1)\supset\cdots\lambda(r)\supset\lambda(r+1)=\emptyset,\\
\mu&=\mu(1)\supset\cdots\mu(r)\supset\mu(r+1)=\emptyset,
\end{split}
\end{equation}
such that
\begin{equation}\label{eq-clm1}
c_{\lambda(i)\cap\mu(i),\lambda(i+1)}^{\lambda(i)}\neq 0~~\text{and}~~c_{\lambda(i)\cap\mu(i),\mu(i+1)}^{\mu(i)}\neq 0,
\end{equation}
and
\begin{equation*}\label{eq-nu1}
\nu^{(1)}_i=|\lambda(i)\cap\mu(i)|
\end{equation*}
for  $1\leq i \leq r$. Then by (\ref{eq-seqlm1}), (\ref{eq-clm1}) and the semigroup property of Littlewood-Richardson coefficients, for all $\ell\geq1$ we have
 \begin{equation}\label{eq-seqlml}
\begin{split}
\ell\lambda&=\ell\lambda(1)\supset\cdots\ell\lambda(r)\supset\ell\lambda(r+1)=\emptyset,\\
\ell\mu&=\ell\mu(1)\supset\cdots\ell\mu(r)\supset\ell\mu(r+1)=\emptyset,
\end{split}
\end{equation}
and
\begin{equation}\label{eq-clml}
c_{\ell\lambda(i)\cap\ell\mu(i),\ell\lambda(i+1)}^{\ell\lambda(i)}\neq 0~~\text{and}~~c_{\ell\lambda(i)\cap\ell\mu(i),\ell\mu(i+1)}^{\ell\mu(i)}\neq 0.
\end{equation}
By (\ref{eq-seqlml}) and (\ref{eq-clml}) we have that $\nu^{(\ell)}$ ($\ell\geq1$) are derived from ($\ell\lambda, \ell\mu$) where $\ell\lambda$, $\ell\mu$ are still rectangular partitions. Moreover, we have
\begin{equation*}
\nu^{(\ell)}=(\nu^{(\ell)}_1,\nu^{(\ell)}_2,\ldots,\nu^{(\ell)}_r),
~\text{where}~\nu^{(\ell)}_i=|\ell\lambda(i)\cap\ell\mu(i)|=\ell\nu^{(1)}_i
\end{equation*}
for $i=1,2,\ldots,r$. That is,
\begin{equation}\label{eq-nul1}
\nu^{(\ell)}=\ell\nu^{(1)}.
\end{equation}

By Proposition \ref{prp-only1}, we have that  $\nu^{(\ell)}$ has maximal lexicographic order in $\Phi(\ell\lambda,\ell\mu)$  for each $\ell\geq1$. Thus  $\nu^{(\ell)}/\ell k$ has
maximal lexicographic order in $\frac{1}{\ell k}\Phi(\ell\lambda,\ell\mu)$ for each $\ell\geq1$. By  (\ref{eq-nul1}) we have that all their normalizations $\nu^{(\ell)}/\ell k$ are equal to $\nu^{(1)}/ k$. Hence  $\nu^{(1)}/ k=\nu/ k$ has the maximal lexicographic order in $\bigcup_{\ell=1}^{\infty}\frac{1}{\ell k}\Phi(\ell\lambda,\ell\mu)$.
Thus, by the density of rational spectra, we have that  $\nu/ k$ has the maximal lexicographic order in $\mathcal{S}(\frac{1}{n}I_n,\frac{1}{m}I_m)$.
\end{proof}


\subsection{Two classes of counterexamples for  Klyachko's conjecture}
In the following two examples, we will show that there exist states $\rho\in \mathcal{C}(\frac{1}{n}I_n,\frac{1}{m}I_m)$ which have the maximal lexicographic spectrum, but they do not have the minimal rank.

\begin{example}\label{exm1-2m}
Let $m \geq 3$ be odd and write it as $2k+1$ for an integer $k$.
Suppose that $\lambda=(m,m)$ and $\mu=(2^m)$. Then it is not hard to see that
 the  partition of strip type derived from $ (\lambda, \mu)$ is
\begin{equation}\label{eq-v2m}
\nu=(4^k,1,1).
\end{equation}
Since $lcm(2,m)=2m$, by Theorem \ref{thm-maxlex} we have
$\nu/2m$ has the maximal lexicographic order in $\mathcal{S}(\frac{1}{2}I_2,\frac{1}{m}I_m)$. Thus, states with maximal lexicographic spectrum in $\mathcal{C}(\frac{1}{2}I_2,\frac{1}{m}I_m)$ have rank $k+2$.

On the other hand, let $\gamma=(4^{k-1},3,3)$. By Theorem 1.6 of \cite{Tew15} and Proposition \ref{prp-trans}, we have that $g(\lambda,\mu;\gamma)=g(\lambda,\mu^t;\gamma^t)=1$. Hence, $\gamma=(4^{k-1},3,3)\in \Phi(\lambda,\mu)$. Hence, there exist states in $\mathcal{C}(\frac{1}{2}I_2,\frac{1}{m}I_m)$ with rank $k+1$. Moreover, by Proposition \ref{prp-rankab} we have that  $k+1$ is the minimal rank for $\mathcal{C}(\frac{1}{2}I_2,\frac{1}{m}I_m)$.
\end{example}

In example above, we see that the rank of states with maximal lexicographic spectrum is close to the minimal rank. However, in the following example we will find that their differences can be large.
\begin{example}\label{exp-2}
Let $\lambda=((n+1)^n)$, $\mu=(n^{n+1})$ and $\nu$ be the partition of strip type derived from $(\lambda, \mu)$. Then we have that $\nu=(\nu_1,\nu_2,\ldots,\nu_{n+1})$ where $\nu_1=n^2$, $\nu_2=\cdots=\nu_{n+1}=1$.
Then by Theorem \ref{thm-maxlex}  the maximal lexicographic spectrum of states in $\mathcal{C}(\frac{1}{n}I_n,\frac{1}{n+1}I_{n+1})$ is
$$\frac{\nu}{n(n+1)}.$$
Hence the rank of those states are $n+1$.

On the other hand, by Proposition 6.9 of \cite{Iken17}, we have that $g(\lambda,\mu;\gamma)=1$, where $$\gamma=(\frac{n(n+1)}{2},\frac{n(n+1)}{2})$$
is a two row partition. Hence, there exist  states with rank 2 in $\mathcal{C}(\frac{1}{n}I_n,\frac{1}{n+1}I_{n+1})$. By Proposition \ref{prp-rankab}, we have that the minimal rank of states in $\mathcal{C}(\frac{1}{n}I_n,\frac{1}{n+1}I_{n+1})$ is 2.
\end{example}

\subsection{On the extremity of states with  maximal lexicographic spectrum}
Let $H$ be a Hermitian matrix. Denote the diagonal entries of $H$ by $D(H)$ which are arranged decreasingly. The well-known Schur Theorem states that $D(H)\trianglelefteq Spec~H$ \cite{Horn,Mars11}. When $D(H)=Spec~H$, by Corollary 4.3.34 and Theorem 4.3.45 in \cite{Horn} we have the following proposition.

\begin{proposition}\label{prp-dhshd}
Let $H$ be a Hermitian matrix.  Then $D(H)=Spec~H$ if and only if $H$ is diagonal.
\end{proposition}
\begin{theorem}
Let $\mathcal{D}(\mathcal{H})$ be the set of density operators on $\mathcal{H}$. Suppose that $\mathcal{C}(\mathcal{H})\subseteq\mathcal{D}(\mathcal{H})$ is a convex subset. If $\rho\in\mathcal{C}(\mathcal{H})$ has the maximal lexicographic spectrum among all other states, then $\rho$ is an extreme point of $\mathcal{C}(\mathcal{H})$.
\end{theorem}
\begin{proof}
Suppose that there exist $\sigma$, $\tau\in\mathcal{C}(\mathcal{H})$ such that
$$\rho=p_1\sigma+p_2\tau,$$
where  $p_1$, $p_2\geq0$ and $p_1+p_2=1$.

Denote $\lambda=Spec~\rho$, $\mu=Spec ~\sigma$, $\eta=Spec~\tau$.
In the following, we don't distinguish between the spectrum and the diagonal matrix with diagonal entries consist of it.
Let  $U$ be the unitary matrix such that $U^*\lambda U=\rho$. Then we have
$$\lambda=p_1U\sigma U^*+p_2U \tau U^*.$$
Let $D(U\sigma U^*)$, $D(U\tau U^*)$ be the diagonal of $U\sigma U^*$ and $U\tau U^*$.
Then we have
$$\lambda=p_1D(U\sigma U^*)+p_2 D(U\tau U^*).$$
By Schur's Theorem we have that
$D(U\sigma U^*)\unlhd \mu$ and $ D(U\tau U^*)\unlhd \eta$. Moreover, since lexicographic order is a refinement of dominance order, we have $D(U\sigma U^*)\leq \mu$ and $ D(U\tau U^*)\leq \eta$.
 Since $\lambda$ has the maximal  lexicographic order, we have
\begin{align*}
\lambda &=p_1D(U\sigma U^*)+p_2 D(U\tau U^*)\leq p_1 \mu +p_2 \eta\\
&\leq p_1 \lambda+p_2\lambda\\
&=\lambda.
\end{align*}
Hence we have
\begin{equation}\label{eq-lamst}
\lambda=p_1D(U\sigma U^*)+p_2 D(U\tau U^*)=p_1 \mu +p_2 \eta.
\end{equation}
Since $D(U\sigma U^*)\leq \mu \leq \lambda$ and $ D(U\tau U^*)\leq \eta\leq \lambda$,
by (\ref{eq-lamst}) we should have that
 \begin{equation*}
 \lambda=\mu=D(U\sigma U^*)~\text{and} ~\lambda=\eta=D(U\tau U^*).
\end{equation*}
Hence by Proposition \ref{prp-dhshd} we have
 \begin{equation*}
 \lambda=\mu=U\sigma U^*~\text{and} ~\lambda=\eta=U\tau U^*,
\end{equation*}
which is equivalent to $\rho=\sigma=\tau$.
\end{proof}

Since $\mathcal{C}(\rho_A, \rho_B)\subseteq \mathcal{D}(\mathcal{H}_A\otimes\mathcal{H}_B)$ is convex, we have the following corollary.
\begin{corollary}
If $\rho\in\mathcal{C}(\rho_A, \rho_B)$ has maximal lexicographic spectrum, then it is an extreme point.
\end{corollary}


\section{Ranks of states in  $\mathcal{C}(\frac{1}{n}I_n, \frac{1}{m}I_m)$}

For $0< n\leq m$, write $m=np+r$ where $p\geq1$ and $0\leq r\leq n-1$.
In this section, we construct states with prescribed ranks in $\mathcal{C}(\frac{1}{n}I_n, \frac{1}{m}I_m)$ which generalizes the construction in \cite{BCI11}.

Suppose that $n\leq m$. Let $|0\rangle$,..., $|m-1\rangle$ denote the standard orthonormal basis of $\mathbb{C}^m$. We define the {\it generalized discrete
Weyl operators} $X,Z_n \in\mathcal{L}(\mathbb{C}^m)$ by
\begin{equation*}
X|i\rangle = |i + 1\rangle;~~ Z_n|i\rangle = \omega^i |i\rangle,
\end{equation*}
where~$\omega^n=1$, $i=0,1,\ldots,m-1$ and the addition is modulo $m$. If $n=m$, these are called the discrete Weyl operators \cite{BCI11}.

For $n\leq m$, the maximal entangled state of $\mathbb{C}^n\otimes\mathbb{C}^m$
are defined by
$$|\psi_{00}\rangle := \frac{1}{\sqrt{n}}\sum_{i=1}^n|i\rangle|i\rangle.$$
In $|i\rangle|i\rangle$ above, without of confusion the first and second $|i\rangle$ represent the standard orthonormal vectors of $\mathbb{C}^n$ and $\mathbb{C}^m$ respectively. Now  we let
$$|\psi_{ij}\rangle := (I_n\otimes X^i Z_n^j) |\psi_{00}\rangle,$$
where $i=0,1,\ldots,m-1$ and $j=0,1,\ldots,n-1$. The following proposition generalizes Lemma 5 of \cite{BCI11}.

\begin{proposition}\label{prp-psiort}
The vectors $|\psi_{ij}\rangle$, for $0\leq i\leq m-1$, $0\leq j \leq n-1$, form an orthonormal basis of $\mathbb{C}^n\otimes\mathbb{C}^m$.
\end{proposition}
\begin{proof}
For any pair of $|\psi_{ij}\rangle$ and $|\psi_{kl}\rangle$, we have
\begin{align*}
\langle\psi_{ij}|\psi_{kl}\rangle &= \langle\psi_{00}|
 (I_n\otimes Z_n^{-j}X^{-i}) (I_n\otimes X^k Z_n^l) |\psi_{00}\rangle\\
 &= \langle\psi_{00}| (I_n\otimes Z_n^{-j} X^{k-i} Z_n^{l}) |\psi_{00}\rangle\\
 &=\frac{1}{n}\sum_{s,s'=1}^{n}\langle ss|(I_n\otimes Z_n^{-j} X^{k-i} Z_n^{l}) |s's'\rangle\\
 &=\frac{1}{n}\sum_{s=1}^{n}\langle s| Z_n^{-j} X^{k-i} Z_n^{l}|s\rangle.
\end{align*}
Since the sum of the first $n$ diagonal entries of $Z_n^t$ and the diagonal of $X^t$ are zeroes for any integer $t\neq0$, we have that $\langle\psi_{ij}|\psi_{kl}\rangle=0$ if $i\neq k$ or $j\neq l$.
\end{proof}

\begin{theorem}\label{thm-rkmmn}
Suppose that $n\leq m$. Then for each $m\leq k\leq mn$ there exist states $\rho\in\mathcal{C}(\frac{1}{n}I_n, \frac{1}{m}I_m)$ with $rank~\rho=k$.
\end{theorem}
\begin{proof}
Suppose that $\sum_{i=0}^{m-1}\sum_{j=0}^{n-1}\lambda_{i,j}=1$ and $\lambda_{i,j}\geq0$ for $i=0,1,\ldots,m-1$ and $j=0,1,\ldots,n-1$.
Let
\begin{equation}\label{eq-lamd}
\Lambda=\left(
                \begin{array}{cccc}
                  \lambda_{0,0} & \lambda_{0,1} & \cdots & \lambda_{0,n-1} \\
                  \lambda_{1,0} & \lambda_{1,1} & \cdots & \lambda_{1,n-1} \\
                  \vdots & \vdots &  & \vdots \\
                  \lambda_{m-1,0} & \lambda_{m-1,1} & \cdots & \lambda_{m-1,n-1}\\
                \end{array}
              \right).
\end{equation}
Denote the row vector of $\Lambda$ by $row(\Lambda)=(\mu_0,\mu_1,\ldots,\mu_{m-1})$ where $\mu_{i}=\sum_{j=0}^{n-1}\lambda_{i,j}$ for $i=0,1,\ldots,m-1$.

For $|\psi_{ij}\rangle$ discussed in Proposition \ref{prp-psiort}, let
$$\rho=\sum_{i=0}^{m-1}\sum_{j=0}^{n-1}
\lambda_{i,j}|\psi_{ij}\rangle\langle\psi_{ij}|.$$
Then $\rho$ is a state of the system $\mathbb{C}^n\otimes\mathbb{C}^{m}$.
Next, we find conditions when $\rho\in\mathcal{C}(\frac{1}{n}I_n, \frac{1}{m}I_m)$ with rank $k$ for $m\leq k\leq mn$.

For $|\psi_{ij}\rangle$, we have
\begin{align*}
tr_{B} (|\psi_{ij}\rangle\langle\psi_{ij}|)&=tr_{B} \left(\frac{1}{n}\sum_{s,s'=0}^{n-1}
|s\rangle\langle s'|\otimes X^{i}Z_{n}^{j}(|s\rangle\langle s'|)Z^{-j}_nX^{-i}\right)\\
&=\frac{1}{n}I_n.
\end{align*}
Hence we have
\begin{equation}\label{eq-trb}
tr_{B} (\rho)=\sum_{i=0}^{m-1}\sum_{j=0}^{n-1}
\lambda_{i,j}tr_{B}\left(|\psi_{ij}\rangle\langle\psi_{ij}|\right)=\frac{1}{n}I_n.
\end{equation}

Since $Z_n^{j}|s\rangle=\omega^{j}|s\rangle$, we have
$$Z_n^{j}|s\rangle\langle s|Z_n^{-j}=\omega^{j}\omega^{-j}|s\rangle\langle s|=|s\rangle\langle s|.$$
Then we have
\begin{align*}
tr_{A} (|\psi_{ij}\rangle\langle\psi_{ij}|)&=tr_{A} \left(\frac{1}{n}\sum_{s,s'=0}^{n-1}
|s\rangle\langle s'|\otimes X^{i}Z_{n}^{j}(|s\rangle\langle s'|)Z^{-j}_nX^{-i}\right)\\
&=\frac{1}{n}\sum_{s=0}^{n-1}
X^{i}Z_{n}^{j}(|s\rangle\langle s|)Z^{-j}_nX^{-i}\\
&=\frac{1}{n}X^{i}\left(\sum_{s=0}^{n-1}
|s\rangle\langle s|\right)X^{-i}.
\end{align*}
Let $P_{ij}=tr_{A} (|\psi_{ij}\rangle\langle\psi_{ij}|)$. Then for $0\leq i\leq m-1$ we have $P_{i,0}=P_{i,1}=\cdots=P_{i,n-1}$ which denoted by $P_i$. Hence we have
\begin{align*}
tr_{A} (\rho)&=\sum_{i=0}^{m-1}\sum_{j=0}^{n-1}
\lambda_{i,j} tr_{A}(|\psi_{ij}\rangle\langle\psi_{ij}|)=\left( \sum_{i=0}^{m-1}\sum_{j=0}^{n-1}
\lambda_{i,j}P_{ij}\right)\\
&= \sum_{i=0}^{m-1}P_{i}\left(\sum_{j=0}^{n-1}
\lambda_{i,j}\right)\\
&= \sum_{i=0}^{m-1}\mu_{i}P_{i},
\end{align*}
where $\mu_i=\sum_{j=0}^{n-1}\lambda_{i,j}$ is the sum of row $i$ of $\Lambda$. Note that
$\sum_{i=0}^{m-1}P_{i}=I_m$. If we let $\mu_i=\frac{1}{m}$ for $i=0,1,\ldots,m-1$, then we have
\begin{align}\label{eq-tra}
tr_{A} (\rho)&=\sum_{i=0}^{m-1}\mu_{i}P_{i}=\frac{1}{m} \sum_{i=0}^{m-1}P_{i}\nonumber\\
&=\frac{1}{m}I_m.
\end{align}

By (\ref{eq-trb}) and (\ref{eq-tra}) we have that if $row(\Lambda)=(\frac{1}{m},
\frac{1}{m},\ldots,\frac{1}{m})$, then $\rho\in\mathcal{C}(\frac{1}{n}I_n, \frac{1}{m}I_m)$ where $\Lambda$ is defined in (\ref{eq-lamd}).
Since the number of nonzero entries of $\Lambda$ is the rank of $\rho$, for $m\leq k\leq mn$ it is not hard to find  $k$ nonzero entries such that $row(\Lambda)=(\frac{1}{m},
\frac{1}{m},\ldots,\frac{1}{m})$.
\end{proof}

\begin{theorem}\label{thm-rkpm}
Suppose that $n\leq m$ and $n|m$. Then for each $\frac{m}{n}\leq k\leq m$ there exist states $\rho\in\mathcal{C}(\frac{1}{n}I_n, \frac{1}{m}I_m)$ with $rank~\rho=k$.
\end{theorem}
\begin{proof}
Let $p=n|m$. Suppose that $\sum_{i=0}^{p-1}\sum_{j=0}^{n-1}\tau_{i,j}=1$ and $\tau_{i,j}\geq0$ for $i=0,1,\ldots,p-1$ and $j=0,1,\ldots,n-1$.
Let
\begin{equation}\label{eq-T}
T=\left(
                \begin{array}{cccc}
                  \tau_{0,0} & \tau_{0,1} & \cdots & \tau_{0,n-1} \\
                  \tau_{1,0} & \tau_{1,1} & \cdots & \tau_{1,n-1} \\
                  \vdots & \vdots &  & \vdots \\
                  \tau_{p-1,0} & \tau_{p-1,1} & \cdots & \tau_{p-1,n-1}\\
                \end{array}
              \right).
\end{equation}
Denote the row vector of $T$ by $row(T)=(\nu_0,\nu_1,\ldots,\nu_{p-1})$ where $\nu_{i}=\sum_{j=0}^{n-1}\tau_{i,j}$ for $i=0,1,\ldots,p-1$.

For $|\psi_{ij}\rangle$ discussed in Proposition \ref{prp-psiort}, let
$$\rho=\sum_{i=0}^{p-1}\sum_{j=0}^{n-1}
\tau_{i,j}|\psi_{in,j}\rangle\langle\psi_{in,j}|.$$
Then $\rho$ is a state of the system $\mathbb{C}^n\otimes\mathbb{C}^{m}$.
Next, we find conditions when $\rho\in\mathcal{C}(\frac{1}{n}I_n, \frac{1}{m}I_m)$ with rank $k$ for $p\leq k\leq m$.

Just as the proof of Theorem \ref{thm-rkmmn} we have
\begin{align*}
tr_{A} (\rho)&=\sum_{i=0}^{p-1}\sum_{j=0}^{n-1}
\tau_{i,j} tr_{A}(|\psi_{in,j}\rangle\langle\psi_{in,j}|)=\left( \sum_{i=0}^{p-1}\sum_{j=0}^{n-1}
\tau_{i,j}P_{in,j}\right)\\
&= \sum_{i=0}^{p-1}P_{in}\left(\sum_{j=0}^{n-1}
\tau_{i,j}\right)\\
&= \sum_{i=0}^{p-1}\nu_{i}P_{in},
\end{align*}
where $P_{in,j}=tr_{A} (|\psi_{in,j}\rangle\langle\psi_{in,j}|)$ ($j=0,1,\ldots,n-1$) are equal and denoted by $P_{in}$ and $\nu_i=\sum_{j=0}^{n-1}\tau_{i,j}$ is the sum of row $i$ of $T$. Observe that
$\sum_{i=0}^{p-1}P_{in}=\frac{1}{n}I_m$, if we let $\nu_i=\frac{1}{p}=\frac{n}{m}$ for $i=0,1,\ldots,p-1$, then we have
\begin{align}\label{eq-tra2}
tr_{A} (\rho)&=\sum_{i=0}^{m-1}\nu_{i}P_{i}=\frac{n}{m} \sum_{i=0}^{m-1}P_{i}\nonumber\\
&=\frac{1}{m}I_m.
\end{align}

Since $tr_{B} (\rho)=\frac{1}{n}I_n$, by (\ref{eq-tra2}) we have that if $row(T)=(\frac{1}{p},
\frac{1}{p},\ldots,\frac{1}{p})$, then $\rho\in\mathcal{C}(\frac{1}{n}I_n, \frac{1}{m}I_m)$ where $T$ is defined in (\ref{eq-T}).
Since the number of nonzero entries of $T$ is the rank of $\rho$, for $p\leq k\leq m$ it is not hard to find  $k$ nonzero entries such that $row(T)=(\frac{1}{p},
\frac{1}{p},\ldots,\frac{1}{p})$.
\end{proof}

\begin{remark}\label{re-attain}
From Theorem \ref{thm-rkpm} we can see that the lower bound of Proposition \ref{prp-rankab} is attainable. Combined with Theorem \ref{thm-rkmmn}, if $n|m$ we have that for each $\frac{m}{n}\leq k\leq mn$ there exists $\rho\in\mathcal{C}(\frac{1}{n}I_n, \frac{1}{m}I_m)$ with $rank~\rho=k$. Generally, it has been shown that there exists $\rho\in\mathcal{C}(\rho_A,\rho_B)$ with rank $k$ if and only if $r\leq k\leq k_A k_B$ where
$r$ is the lowest rank of states in $\mathcal{C}(\rho_A,\rho_B)$ \cite{CKLi17}.

For $n\nmid m$, by the discussion in Example \ref{exp-2} we have that if $m=n+1$ then there exist states in $\mathcal{C}(\frac{1}{n}I_n, \frac{1}{n+1}I_{n+1})$ with rank $\lceil \frac{n+1}{n}\rceil=2$. Thus, the lower bound of Proposition \ref{prp-rankab} is also attainable. By Theorem 2.2 of \cite{CKLi17} we have that there exist states in $\mathcal{C}(\frac{1}{n}I_n, \frac{1}{n+1}I_{n+1})$ with ranks from 2 to $n(n+1)$.
When $n\nmid m$ it is interesting to give the construction of states with ranks from
$\lceil \frac{m}{n}\rceil$ to $mn$. Recently, in \cite{BLRR18} the authors discussed the construction of locally maximally entangled state of  multipart quantum systems. By their results, we can decide whether there exist states with spectra $(\frac{1}{k},\frac{1}{k},...,\frac{1}{k})$ in $\mathcal{C}(\frac{1}{n}I_n,\frac{1}{m}I_m)$ where $1\leq k\leq mn$. When $k=2$, they gave an explicit construction of such states.

Suppose that $n|m$ and $p=\frac{m}{n}$. Then in Theorem \ref{thm-maxlex} we have $k=m$, $a=p$ and $b=1$. If $\nu$ is the partition of strip type derived from $\lambda=(p^n)$ and $\mu=(1^m)$,  then we have $\nu=(n^p)$. Thus the maximal lexicographic spectrum for states in $\mathcal{C}(\frac{1}{n}I_n,\frac{1}{m}I_m)$ is $\nu/m=(\frac{1}{p},\frac{1}{p},...,\frac{1}{p})$. So the rank of these states is $p$.
Comparing with Proposition \ref{prp-rankab}, when $n|m$ we can see that the rank of states with maximal lexicographic spectrum  is minimal in $\mathcal{C}(\frac{1}{n}I_n,\frac{1}{m}I_m)$. Thus if $n|m$ then Klyachko's conjecture is true for states in $\mathcal{C}(\frac{1}{n}I_n,\frac{1}{m}I_m)$. In Theorem \ref{thm-rkpm} we give a construction of such states.
\end{remark}

%

The geometric complexity theory program is an approach to separate algebraic complexity classes.  Rectangular Kronecker coefficients play an important role in geometric complexity theory \cite{BCI11, Iken17}. For example, it can be used to prove the lower bounds of determinantal complexity. By the construction in Theorem \ref{thm-rkmmn} and \ref{thm-rkpm} and the proof of  main results in \cite{BCI11}, we can get nonzero stretched Kronecker coefficients for a pair of different rectangular partitions. For example, just as Theorem 1 in \cite{BCI11} we have the following corollary.
\begin{corollary}
Suppose that $n|m$, $p=\frac{m}{n}$ and $a \in \mathbb{N}$. Let $\lambda=(a^m)$ and $\mu=((pa)^n)$. For each partition $\nu\vdash ma$ if there exists a $p\times n$ nonnegative matrix $A$ with constant row sum $na$ such that its nonzero entries consist of all parts of $\nu$, then there exists a  stretching factor $k\in \mathbb{N}$ such that $g(k\lambda,k\mu;k\nu)\neq0$.
\end{corollary}

\section*{Acknowledgments}
We would like to thank the referee for many helpful comments and suggestions.


\bibliographystyle{amsplain}


%

\end{document}